\documentclass{amsart}
\usepackage{amsfonts}
\usepackage{amsmath,amsthm}
\usepackage{amsfonts,amssymb}
\usepackage{mathrsfs}
\usepackage{enumerate}

\newtheorem{thm}{Theorem}[section]

\newtheorem{lem}[thm]{Lemma}

\theoremstyle{remark}
\newtheorem{rem}[thm]{Remark}

\numberwithin{equation}{section}

\newcommand{\al}{\alpha}

\def\vz{\varepsilon}
\def\oz{\omega}

\def\az{\alpha}

\def\Gz{\Gamma}

\def\({\Bigl(}
\def \){ \Bigr)}

 \def\k{{\kappa}}

 \def\RR{{\mathbb R}}

\def\x{{\bf x}}

\def\k{{\bf k}}
\def\kk{{\bf k}}

\def\y{{\bf y}}

\begin{document}
\def\RR{\mathbb{R}}
\def\Exp{\text{Exp}}
\def\FF{\mathcal{F}_\al}

\title[] {Preasymptotics and asymptotics of approximation numbers of anisotropic Sobolev embeddings}

\author[]{Jia Chen, Heping Wang} \address{School of Mathematical Sciences, Capital Normal
University, Beijing 100048,
 China.}
\email{ jiachencd@163.com;\ \  \ wanghp@cnu.edu.cn.}


\thanks{
 Supported by the
National Natural Science Foundation of China (Project no.
11271263),
 the  Beijing Natural Science Foundation (1132001) and BCMIIS
 }

\begin{abstract} In this paper, we  obtain the preasymptotic and asymptotic
behavior and  strong equivalences  of the approximation numbers of
the embeddings from the anisotropic Sobolev spaces $W_2^{\bf
R}(\Bbb T^d)$  to $L_2(\Bbb T^d)$.  We also get the preasymptotic
 behavior   of the approximation numbers
of the embeddings from the limit spaces $W_2^{\infty}(\Bbb T^d)$
of the anisotropic Sobolev spaces $W_2^{\bf R}(\Bbb T^d)$  to
$L_2(\Bbb T^d)$. We show that both the above  embedding problems
are intractable and do not suffer from the curse of
dimensionality.

\end{abstract}

\maketitle
\input amssym.def

\section{Introduction}

Let $X,Y$ be two Banach spaces. For a bounded linear operator
$T:X\rightarrow Y$, the approximation numbers of it are defined as
\begin{align*}\label{1.1}
a_n(T:X\rightarrow Y):&=\inf_{rank A< n}\sup_{\|x|X\|\leq 1}\|Tx-Ax|Y\|\notag\\
&=\inf_{rank A< n}\|T-A:X\rightarrow Y\|,\ \ n\in \Bbb N_+.
\end{align*}
They describe the best approximation of $T$ by finite rank
operators.

Recently, the papers \cite{KSU1, KSU2} considered  sharp constants
of the approximation numbers and tractability of the embeddings
from the Sobolev spaces of isotropic  and dominating mixed
(fractional) smoothness for various equivalent norms including the
classical one on the $d$-dimensional torus $\Bbb T^d$ to the $L_2$
space. Specially, the authors of \cite{KSU1, KMU} and \cite{KSU2}
obtained the asymptotic and preasymptotic behavior  of the
approximation numbers
 of the isotropic Sobolev embeddings and the mixed order Sobolev embeddings. We note that there are
another Sobolev spaces--anisotropic Sobolev spaces, which may be
viewed as generalization of the isotropic Sobolev spaces. In this
paper, we investigate the preasymptotic and asymptotic behavior of
the approximation numbers of the anisotropic Sobolev embeddings
\begin{equation}\label{1.1}I_d: \,W_2^{\bf R}(\Bbb
T^d)\longrightarrow L_2(\Bbb T^d), \ \ {\bf
R}=(R_1,R_2,\dots,R_d)\in \Bbb R_+^d,\ \end{equation} and the
limit space embeddings of the anisotropic Sobolev spaces
\begin{equation}\label{1.1-11}I_d\:: W_2^{\infty}(\Bbb
T^d)\rightarrow L_2(\Bbb T^d), \end{equation} where $I_d$ is the
identity (embedding) operator, and $W_2^{\bf R}(\Bbb T^d)$ and
$W_2^{\infty}(\Bbb T^d)$ are the anisotropic Sobolev spaces and
their limit spaces whose definitions
  will be
given in Sections 2.2 and 2.3.

In recent years, there has been an increasing interest in
multivariate computational problems which are defined on classes
of functions depending on $d$ variables, since many problems,
e.g., in finance or quantum chemistry, are modeled in associated
function spaces on high-dimensional domains. So far, many authors
have contributed to the subject, see for instance the monographs
by Temlyakov \cite{T} and the references therein. In \cite[Chapter
2, Theorems 4.1, 4.2]{T}, the following two-sided estimate can be
found:
\begin{equation}\label{1.2}c({\bf R},d)n^{-g({\bf R})}\leq a_n(I_d
: W_2^{\bf R}(\Bbb T^d)\rightarrow L_2(\Bbb T^d))\leq C({\bf
R},d)n^{-g({\bf R})},\: n\in \Bbb N_+, \end{equation} where
$g({\bf R})=\frac{1}{1/R_1+...+1/R_d}$, and the constants $c({\bf
R},d)$ and $C({\bf R},d)$,  only depending on $d$ and ${\bf R}$,
were not explicitly determined. When $R_1=R_2=\dots=R_d=s>0$,
$W_2^{\bf R}(\Bbb T^d)$ recedes to the natural isotropic Sobolev
space $H^{s,2s}(\Bbb T^d)$, $g({\bf R})=s/d$, where  the
definition of the Sobolev space $H^{s,r}(\Bbb T^d)\ (s>0,\ 0<r \le
\infty)$ will be given in Section 2.1. When $\min\{R_1,\dots,
R_d\}\to+\infty$, $W_2^{\bf R}(\Bbb T^d)$ tends to $W_2^{\infty
}(\Bbb T^d)$ in the sense of the set limit.

 Our main focus is to clarify, for
arbitrary but fixed ${\bf R}$, the dependence of these constants
on $d$. Surprisingly, for sufficiently large $n$, say $n>3^d$, it
turns out that the optimal constants decay polynomially in $d$,
i.e.,
$$c({\bf R},d)\asymp C({\bf R},d)\asymp d^{-1/2},$$where
equivalent constants depend only on $\max\{R_1,\dots,R_d\}$ and  $
\min\{ R_1,\dots,R_d\}$. Specially, we obtain the strong
equivalence of the approximation numbers $ a_n(I_d : W_2^{\bf
R}(\Bbb T^d)\rightarrow L_2(\Bbb T^d))$ as $n\to\infty$.

For small $n, 1\le n\le 3^d$, we also determine explicitly how the
approximation numbers $ a_n(I_d : W_2^{\bf R}(\Bbb T^d)\rightarrow
L_2(\Bbb T^d))$ and $ a_n(I_d : W_2^{\infty}(\Bbb T^d)\rightarrow
L_2(\Bbb T^d))$ behave preasymptotically. We emphasis that the
preasymptotic behavior of $ a_n(I_d : W_2^{\bf R}(\Bbb
T^d)\rightarrow L_2(\Bbb T^d))$ is  completely different from its
asymptotic behavior. However, the preasymptotic behavior of $
a_n(I_d : W_2^{\infty}(\Bbb T^d)\rightarrow L_2(\Bbb T^d))$
coincides with the one of  $ a_n(I_d : W_2^{\bf R}(\Bbb
T^d)\rightarrow L_2(\Bbb T^d))$. Our main results  will be in full
analogy with those of \cite{KSU1, KMU}
 for the case of $H^{s,2s}(\Bbb T^d)$.

 Finally we consider weak
tractability results for the approximation problem of the
anisotropic Sobolev embeddings \eqref{1.1} and \eqref{1.1-11}.
 Based on   results of \cite{KSU1, KMU} and  preasymptotic behavior of $
a_n(I_d : W_2^{\infty}(\Bbb T^d)\rightarrow L_2(\Bbb T^d))$, we
show that the approximation problems \eqref{1.1} and
\eqref{1.1-11} are intractable and do not suffer from the curse of
dimensionality.

The paper is organized as follows. In Section 2 we give
definitions  of  the isotropic Sobolev spaces with various
equivalent norms,
  the anisotropic Sobolev spaces, the limit spaces $W_2^\infty(\Bbb T^d)$,  and   tractability, and then state out    main results.
  In the final Section 3 we prove the main results.

\section{Preliminaries and main results}

\subsection{Isotropic Sobolev spaces}

\

 For $d\in\Bbb N_+,\  \x =(x_1,\dots,
x_d) \in \Bbb R^d$,  we set $|\x|_p=\Big(\sum_{j=1}^\infty
|x_j|^p\Big)^{1/p}$ for $0 <p <\infty$, and
$|\x|_\infty=\max_{1\le j\le d}|x_j|$ for $p=\infty$. In what
follows $\Bbb T$ denotes the torus, i.e., $\Bbb T =[0, 2\pi]$,
where the endpoints of the interval are identified, and $\Bbb T^d$
stands for the $d$-dimensional torus. We equip $\Bbb T^d$ with the
normalized Lebesgue measure $(2\pi)^{-d}d\x$. Consequently,
$\{e^{i\k\x}:\ \k\in \Bbb Z^d\}$ is an orthonormal basis in
$L_2(\Bbb T^d)$, where $\k\x =\sum_{j=1}^d k_jx_j$. The Fourier
coefficients of a function $f\in L_1(\Bbb T^d)$ are defined as
$$\hat f(\k)=(2\pi)^{-d}\int_{\Bbb T^d}f(\x)e^{-i\k\x}d\x,\ \  \k\in \Bbb
Z^d.$$

For $0 <s <\infty$ and $0< r\le \infty$ we denote by $H^{s,r}(\Bbb
T^d)$ the isotropic Sobolev space formed by all $f\in L_2(\Bbb
T^d)$ having a finite norm
\begin{equation}\label{e1.1}\|f|H^{s,r}(\Bbb T^d )\|=\big(\sum_{\k\in \Bbb Z^d}(1+\sum_{j=1}^d|k_j|^r)^{2s/r}|\hat
f(\k)|^2\big)^{1/2}.
\end{equation}
Clearly, for the fixed isotropic smoothness index $s>0$, all these
norms are equivalent, whence all spaces $H^{s,r}(\Bbb T^d)$ with
$0< r \le \infty$ coincide. The superscript $r$ just indicates
which norm we are considering. For integer smoothness $s = m\in
\Bbb N_+$, the most natural norms are those with $r=2$ and $r=2m$.
Indeed, let $D^\alpha f$ be the distributional derivative of $f$
of order $\alpha = (\alpha _1 ,...,\alpha _d)$.  The natural
isotropic Sobolev space $W_2^m(\Bbb T^d)$  is defined as
$$W_2^m(\Bbb T^d):=\Big\{f\ |\
\|f|W_2^m(\Bbb T^d)\|=\Big(\|f|L_2(\Bbb
T^d)\|^2+\sum_{j=1}^d\|\frac{\partial ^mf}{\partial
x_j^m}|L_2(\Bbb T^d)\|^2\Big)^{1/2}<\infty\Big\}. $$ If $r=2m$,
one has  equality
$$\|f|H^{m,2m}(\Bbb T^d)\|=\|f|W_2^m(\Bbb T^d)\|.$$

The classical isotropic Sobolev space $H^m(\Bbb T^d)$  is defined
as
$$H^m(\Bbb T^d):=\Big\{f\ |\ \|f|H^m(\Bbb T^d)\|=\Big(\sum_{\alpha\in \Bbb N^d, \ |\alpha |_1\leq m}\|D^\alpha f|L_2(\Bbb T^d)\|^2\Big)^{1/2}<\infty\Big\}. $$
As shown in \cite{KSU1}, one has
\begin{equation}\label{2.1-1}\frac{1}{\sqrt{m!}}\|f|H^{m,2}(\Bbb T^d)\|\leq \|f|H^m(\Bbb
T^d)\|\leq \|f|H^{m,2}(\Bbb T^d)\|.\end{equation} Note that the
equivalence constants depend only on the smoothness index $m$, but
not on the dimension $d$.

\subsection{Anisotropic Sobolev spaces}

\

For  $f\in {L}_{2}(\Bbb T^d)$  and ${\bf R}=(R_{1},\dots,R_{d})\in
\Bbb R_+^{d}$,\ \ let
 $\frac{\partial ^{R_{j}}}{\partial x_{j}^{R_{j}}}f$ be  the $R_j$-order
 partial derivative of $f$ with respect to $x_j$ in the sense of Weyl
 defined by $$\frac{\partial ^{R_{j}}}{\partial x_{j}^{R_{j}}}f({\bf x}):=\sum_{\kk\in\, {\Bbb Z}^{d}}(
 ik_{j})^{R_{j}}\hat{f}({\bf k})e^{i\,{\bf k}{\bf x}},\ \
 (ik_{j})^{R_{j}}=|k_{j}|^{R_j}\exp\,(\frac{R_j\pi i}{2}{\rm sign}\,k_{j}).$$
     The anisotropic Sobolev space
 ${{W}}_{2}^{{\bf R}}(\Bbb T^d)$ is defined by
 $${{ W}}_{2}^{{\bf R}}(\Bbb T^d):=\Big\{f\in {L}_{2}(\Bbb T^d)\ \Big |\ \frac{\partial ^{R_{j}}}
 {\partial x_{j}^{R_{j}}}f\in L_{2}(\Bbb T^d),\ \ j=1,\dots,d\Big\}$$
 with norm \begin{align*}\|f|{{{ W}}_{2}^{{\bf R}}(\Bbb T^d)}\|&=
  \Big(\|f|L_2(\Bbb T^d)\|+\sum_{j=1}^{d}\Big\|\frac{\partial ^{R_{j}}}{\partial
  x_{j}^{R_{j}}}f\,\big|
  L_{2}(\Bbb T^d)\Big\|^{2}\Big)^{1/2}\\&=
  \Big(\sum_{{\bf k}\in \Bbb Z^d}\big(1+\sum_{j=1}^d|k_j|^{2R_j}\big)|\hat f({\bf k})|^2\Big)^{1/2}.\end{align*}It is well known that
 ${{ W}}_{2}^{{\bf R}}(\Bbb T^d)$\ is a Hilbert space. If
 $R_1=\dots=R_d=s>0$, the anisotropic Sobolev space ${{ W}}_{2}^{{\bf R}}(\Bbb
 T^d)$ recedes to the isotropic Sobolev space $H^{s,2s}(\Bbb
 T^d)$. Hence, the anisotropic Sobolev spaces are generalization
 of the isotropic Sobolev spaces. For integer smoothness ${\bf R} = {\bf m}=(m_1,\dots,m_d)\in
\Bbb N_+^d$, ${{W}}_{2}^{{\bf m}}(\Bbb T^d)$ is just the natural
anisotropic Sobolev space. The classical anisotropic Sobolev space
$H^{{\bf m}}(\Bbb T^d)$ is defined by
$$H^{\bf m}(\Bbb T^d):=\Big\{f\ |\ \|f|H^{\bf m}(\Bbb T^d)\|=
\Big(\sum_{\alpha\in \Bbb N^d, \
\frac{\alpha_1}{m_1}+\cdots+\frac{\alpha_d}{m_d}\le 1}\|D^\alpha
f|L_2(\Bbb T^d)\|^2\Big)^{1/2}<\infty\Big\}. $$ When
$m_1=\cdots=m_d=m$, the natural and classical anisotropic Sobolev
spaces ${{W}}_{2}^{{\bf m}}(\Bbb T^d)$ and $H^{{\bf m}}(\Bbb T^d)$
recede to the natural and classical (isotropic) Sobolev spaces
${{W}}_{2}^{{ m}}(\Bbb T^d)$ and $H^{{ m}}(\Bbb T^d)$. It is
easily seen that
$$ \|f|{{{ W}}_{2}^{{\bf m}}(\Bbb T^d)}\|\le \|f|H^{\bf m}(\Bbb
T^d)\|\le C_{d,{\bf m}} \|f|{{{ W}}_{2}^{{\bf m}}(\Bbb T^d)}\|,$$
where $C_{d,{\bf m}}$ is a positive constant depending only on $d$
and ${\bf m}$. Due to \eqref{2.1-1}, one may conjecture that for
general ${\bf m}\in \Bbb N_+^d$,
$$c_{\bf m}\|f|H^{\bf m}(\Bbb
T^d)\|^2\le \sum_{{\bf k}\in\Bbb
Z^d}\big(1+\sum_{j=1}^d|k_j|^{2m_j/p_0}\big)^{p_0} |\hat f({\bf
k})|^2\le C_{\bf m}\|f|H^{\bf m}(\Bbb T^d)\|^2$$ holds for some
$p_0$, $c_{\bf m}, \ C_{\bf m}$, where the constants  $c_{\bf m}$
and $C_{\bf m}$ depend only on ${\bf m}$. However, this is not
true. Indeed, for any positive numbers $p_0, c_{\bf m}, \ C_{\bf
m}$, the above inequality is not valid.  In this paper, we
consider  only the space  ${{ W}}_{2}^{{\bf R}}(\Bbb
 T^d)$.

 \subsection{Limit space of anisotropic (or isotropic) Sobolev spaces}

 \

 Let $X_j$ be Banach spaces. We define
 $\bigwedge _{j=1}^\infty X_j$ to be the space of all elements of $\bigcap_{j=1}^\infty X_j$ for which $\sup_{1\le j<\infty}
\|x|X_j\|<\infty$,   i.e.,
$$\bigwedge _{j=1}^\infty X_j=\big\{x \in\bigcap_{j=1}^\infty X_j\
\big | \ \|x|\bigwedge _{j=1}^\infty X_j\|=\sup_{1\le j<\infty}
\|x|X_j\|<\infty\big\}.$$

In this paper, we consider the space  $W_2^\infty(\Bbb
T^d)=\bigwedge _{{\bf m}\in \Bbb N_+^d} W_2^{\bf m}(\Bbb T^d)$.
Clearly, $W_2^\infty(\Bbb T^d)$ may be viewed as the limit space
of the anisotropic Sobolev spaces $W_2^{\bf R}(\Bbb T^d)$. Indeed,
when $\min\{R_1,\dots, R_d\}\to+\infty$, $W_2^{\bf R}(\Bbb T^d)$
tends to $W_2^{\infty }(\Bbb T^d)$ in  the sense of the set limit.

Note that if $f\in W_2^\infty(\Bbb T^d)$ and  $\k\in \Bbb Z^d,\
|k_i|\ge 2$ for some $i\in \{1,2,\dots,d\}$, then $\hat f(\k)=0$.
It follows that
\begin{align*}W_2^\infty(\Bbb T^d)=\Big\{ f\ |\
f(x)&=\sum_{\k\in\{-1,0,1\}^d}\hat f(\k)e^{-i\k\x}, \\  \|f
|W_2^\infty(\Bbb
T^d)\|&=\Big(\sum_{\k\in\{-1,0,1\}^d}(1+\sum_{j=1}^d|k_j|)|\hat
f(\k)|^2\Big)^{1/2}\Big\}.\end{align*}

Clearly, $W_2^\infty(\Bbb T^d)$ is a Hilbert space with $\dim
\big(W_2^\infty(\Bbb T^d)\big)=3^d$. We also note that
$$W_2^\infty(\Bbb T^d)=\bigwedge _{m=1}^\infty W_2^m(\Bbb T^d).$$

\begin{rem} We note that  the space $\bigwedge _{m=1}^\infty H^m(\Bbb T^d)$ is just the space of
constants and its dimension is $1$. This means  the investigation
of $\bigwedge _{m=1}^\infty H^m(\Bbb T^d)$ is meaningless.

\end{rem}

 \subsection{General notions of tractability}

 \

Recently, there has been an increasing interest in $d$-variate
computational problems  with large or even huge $d$. Such problems
are usually solved by algorithms that use finitely many
information operations.  The information complexity $n(\vz, d)$ is
defined as the minimal number of information operations which are
needed to find an approximating solution to  within an error
threshold $\vz$. A central issue is the study of how the
information complexity depends on $\vz^{-1}$ and $d$. Such problem
is called the tractable problem. Nowadays tractability of
multivariate problems is a very active research area (see
\cite{NW1, NW2,NW3} and the references therein).

 Let
${H_d}$ and ${G_d}$ be two sequences of Hilbert spaces and for
each $d\in \Bbb N_+$,  $F_d$ be the unit ball of  $H_d$. Assume a
sequence of bounded linear operators (solution operators)
$$S_d : H_d\rightarrow G_d$$ for all $d \in \Bbb N_+$.
For $n\in \Bbb N_+$ and $f\in F_d$, $S_d f$ can be approximated by
algorithms
$$A_{n,d}(f)=\Phi _{n,d}(L_1(f),...,L_n(f)),$$
where  $L_j,\ j=1,\dots,n$ are continuous linear functionals on
$F_d$ which are called  general information, and $\Phi _{n,d} :
\Bbb R^n\rightarrow G_d$ is an arbitrary mapping. The worst case
error $e(A_{n,d})$ of the algorithm $A_{n,d}$ is defined as
$$e(A_{n,d})=\sup_{f\in F_d}
\|S_d(f)-A_{n,d}(f)\|_{G_d}.$$  Furthermore, we define the $n$th
minimal worst-case error as
$$e(n,d )=\inf_{A_{n,d}}e(A_{n,d}),$$
where the infimum is taken over all algorithms using $n$
information operators $L_1,L_2,...,L_n$. For $n=0$, we use
$A_{0,d}=0$. The error of $A_{0,d}$ is called the initial error
and is given by
$$e(0,d )=\sup_{f\in F_d}\|S_d f\|_{G_d}.$$
The $n$th minimal worst-case error $e(n,d)$ with respect to
arbitrary  algorithms and general information in the Hilbert
setting is just the $n+1$-approximation number $a_{n+1}(S_d:H_d\to
G_d)$ (see \cite[p. 118]{NW1}), i.e.,
$$e(n,d)=a_{n+1}(S_d:H_d\to G_d).$$

    In this paper, we consider the embedding operators $S_d=I_d$
(formal identity operators) with $e(0,d)=\|I_d\|=1$. In other
words, the normalized error criterion and the absolute error
criterion coincide.
  For $\varepsilon \in (0,1)$ and $d\in \Bbb N_+$, let $n(\varepsilon, d)$ be the information
  complexity which is defined as the minimal number of continuous linear functionals which are necessary to obtain
  an $\varepsilon -$approximation of $I_d$, i.e.,
$$n(\varepsilon ,d)=\min\{n\,|\,e(n,d)\leq \varepsilon \}.$$

We say that the approximation problem  is called weakly tractable,
if
\begin{equation*}\label{t1.2}
\lim_{\varepsilon ^{-1}+d\rightarrow \infty }\frac{\ln
n(\varepsilon ,d)}{\varepsilon ^{-1}+d}=0,
\end{equation*}i.e., $n(\vz,d)$ neither depends exponentially on $1/\vz$ nor on $d$. Otherwise, the approximation problem is called
intractable.

 If there exist two constants $C,t>0$ such that for
all $d\in \Bbb N_+, \ \varepsilon
 \in(0,1)$,
\begin{equation*}\label{t1.4}
n(\varepsilon ,d)\leq C\exp(t(1+\ln\varepsilon ^{-1})(1+\ln d)),
\end{equation*} then the  approximation problem is quasi-polynomially tractable.

If  there exist positive numbers $C, \varepsilon _0, \gamma $ such
that for all $0<\varepsilon\leq \varepsilon _{0}$ and infinitely
many $d\in \Bbb N_+$,
\begin{equation*}\label{1.7}
n(\varepsilon ,d)\geq C(1+\gamma )^{d},
\end{equation*}  then we
say that the approximation problem suffers from the curse of
dimensionality.

Recently, Siedlecki and  Weimar introduced the notion of
$(\alpha,\ \beta)$-weak tractability in \cite{SW}. If for some
fixed $\alpha,\ \beta>0$ it holds
\begin{equation*}
\lim_{\varepsilon ^{-1}+d\rightarrow \infty }\frac{\ln
n(\varepsilon ,d)}{(\varepsilon ^{-1})^{\alpha }+d^{\beta }}=0,
\end{equation*}
then the approximation problem is called $(\alpha,\ \beta)$-weakly
tractable.

 Clearly,  $(1,1)$-weak tractability is  just weak
tractability,  whereas the approximation problem is uniformly
weakly tractable  if it is $(\alpha,\beta)$-weakly tractable  for
all positive $\alpha$ and $\beta$ (see \cite{S}). Also, if the
approximation problem suffers from the curse of dimensionality,
then for any $\alpha>0, \ 0<\beta\le1,$ it is not
$(\alpha,\beta)$-weakly tractable.

\subsection{Main results}

\

In the paper,  we  discuss  the asymptotic and preasymptotic
behavior of the approximation numbers and tractability of the
embeddings from  $W_2^{\bf R}(\Bbb T^d)$  and $W_2^{\infty}(\Bbb
T^d)$ to $L_2(\Bbb T^d)$. For ${\bf R}=(R_1,\dots,R_d)\in\Bbb
R^d_+$ and $t>0$, denote by
$$B_{\bf R}^d(t):=\{\x\in \Bbb R^d:\sum_{j=1}^{d}|x_j|^{R_j}\leq
t\}$$ the generalized  ball in $\Bbb R^d$. We write $B_{\bf R}^d$
instead of $B_{\bf R}^d(1)$ for brevity.  When
$R_1=\dots=R_d=r>0$, $B_{\bf R}^d$ recedes to the unit ball
$B_r^d$ in $\Bbb R^d$ with respect to the (quasi)-norm
$|\cdot|_r$. It follows from \cite{Wa} that the volume of the unit
ball $B^d_{\bf R}$ is
\begin{equation}\label{2.9}
\text{vol}(B^d_{\bf R}):=\text{vol}\{\x\in \Bbb
R^d:\sum_{j=1}^d|x_j|^{R_j}\leq 1\}=2^d\frac{\Gamma(1 +
1/R_1)\cdots \Gamma(1 + 1/R_d)}{\Gamma(1 + 1/R_1+\cdots+1/R_d)},
\end{equation}
where $\Gz(x)=\int_0^\infty t^{x-1}e^{-t}dt$ is the Gamma
function. Specially, we have
\begin{equation}\label{2.10}\text{vol}(B^d_r):=\text{vol}\{\x\in
\Bbb R^d:\sum_{j=1}^d|x_j|^{r}\leq 1\}=2^d\frac{\Gamma(1 +
1/r)^d}{\Gamma(1 + d/r)}.\end{equation}

In \cite[Theorem 1.1]{KMU} or \cite{KSU1}, the authors used
entropy number argument, and combinatorial and volume arguments to
 obtain the preasymptotic and asymptotic behavior of the
approximation numbers $a_n(I_d: H^{s,r}(\Bbb T^d)\rightarrow
L_2(\Bbb T^d))$.   For $0<r\le \infty $ and $s>0$, they got
\begin{align}\label{2.3.1}
&a_n(I_d: H^{s,r}(\Bbb T^d)\rightarrow L_2(\Bbb T^d))\asymp
_{s,r}\left\{\begin{matrix}
1, & \ \  1\le n\leq d,\\
 \Big(\frac{\log(1+\frac{d}{\log n})}{\log n}\Big)^{s/r},&\ \  d\le n\le 2^d, \\
 d^{-s/r}n^{-s/d},&\ \  n\ge 2^d,
\end{matrix}\right.
\end{align}
where $\log x=\log_2 x$, $A\asymp B$ means that there exist two
constants $c$ and $C$ which are called the equivalent constants
such that $c A\le B\le C A$, and $\asymp_{s,r}$ indicates that the
equivalent constants depend only on $s, r$. It is remarkable that
the equivalent constants in the above preasymptotics and
asymptotics depend not on $d$ and $n$. For $n\to \infty$,  the
equivalent constants in the lower and upper bound even converge.
Indeed, in \cite{KSU1} (see also \cite[Proposition 4.1]{CKS}) the
authors obtained the following result:
\begin{align}\label{2.6-111}\lim_{n\rightarrow \infty}n^{s/d}a_n\big(I_d : H^{s,r}(\Bbb T^d)\rightarrow L_2(\Bbb T^d)\big)=(\mathrm{vol}(B_r^d))^{s/d}\asymp_{s,r}
d^{-s/r}.\end{align}

 In this paper, we generalize the above results to  $W_2^{\bf R}(\Bbb T^d)$ and  $W_2^{\infty}(\Bbb T^d)$. We use the volume argument to get the asymptotic behavior of
 $ a_n(I_d:\ W_2^{{\bf R}}(\Bbb T^d)\rightarrow L_2(\Bbb T^d))$. Note that our
 generalization is not trivial since we need to establish a new
 inequality instead of the triangle inequality of the (quasi)-norm
$|\cdot|_r$. We use \eqref{2.3.1} to obtain the preasymptotic
behavior of $ a_n(I_d:\ W_2^{{\bf R}}(\Bbb T^d)\rightarrow
L_2(\Bbb T^d))$. We use combinatorial argument to
 obtain the preasymptotic  behavior of the
approximation numbers $a_n(I_d: W^\infty _2(\Bbb T^d)\rightarrow
L_2(\Bbb T^d))$.  Finally we  give
 the tractability results about the approximation problems  $I_d\:: W_2^{\bf R}(\Bbb T^d)\rightarrow L_2(\Bbb T^d)$ and
  $I_d\:: W_2^{\infty}(\Bbb T^d)\rightarrow L_2(\Bbb T^d)$. Our main results  are formulated as
 follows.

\begin{thm}\label{t2.1}
Let ${\bf R}=(R_1,...,R_d)\in \Bbb R_+^d$,
$u=\max\{R_1,\dots,R_d\}$, $\ v=\min\{ R_1,\dots,R_d\}$, $g({\bf
R})=\frac{1}{1/R_1+...+1/R_d}$.  Then we have
\begin{align}\label{2.6-1}
&a_n(I_d:\ W_2^{{\bf R}}(\Bbb T^d)\rightarrow L_2(\Bbb T^d))\asymp
_{u,v}\left\{\begin{matrix}
1, & \ \  1\le n\leq d,\\
 \Big(\frac{\log(1+\frac{d}{\log n})}{\log n}\Big)^{1/2},&\ \  d\le n\le 3^d, \\
 d^{-1/2}n^{-g({\bf R})},&\ \  n\ge 3^d.
\end{matrix}\right.
\end{align}

\end{thm}

\begin{rem}Let ${\bf R}=(R_1,\dots,R_d)\in\Bbb R_+^d$,  $u=\max\{R_1,\dots,R_d\}$,  $v=\min\{
R_1,\dots,R_d\}$, and  $g({\bf R})=\frac{1}{1/R_1+...+1/R_d}$.
Then for all $n\in\Bbb N_+$, the above theorem gives the exact
decay rate in $n$ and the exact order of the constants with
respect to $d$. We emphasis that the equivalent constants in
\eqref{2.6-1} are independent of $d$ and $n$. Note that when
$R_1=\dots=R_d=s>0$, $g({\bf R})=s/d$,  the anisotropic Sobolev
space ${{ W}}_{2}^{{\bf R}}(\Bbb
 T^d)$ recedes to the isotropic Sobolev space $H^{s,2s}(\Bbb
 T^d)$,  and \eqref{2.6-1} recedes to \eqref{2.3.1} with $r=2s$. \end{rem}

\begin{thm}\label{t2.2}
Let ${\bf R}=(R_1,...,R_d)\in \Bbb R_+^d$. Then we have
\begin{equation}\label{2.6-11}\lim_{n\rightarrow \infty}n^{g({\bf R})}a_n(I_d\:: W_2^{\bf
R}(\Bbb T^d)\rightarrow L_2(\Bbb T^d))=(\mathrm{vol}(B_{2{\bf
R}}^d))^{g({\bf R})},\end{equation} where $g({\bf
R})=\frac{1}{1/R_1+...+1/R_d}$.
\end{thm}

\begin{rem}Let ${\bf R}=(R_1,...,R_d)\in \Bbb R_+^d$. Then for  $g({\bf R})=\frac{1}{1/R_1+...+1/R_d}>1/2$,
 the space $W_2^{\bf R}(\Bbb T^d)$ can compactly embed into $L_\infty(\Bbb T^d)$ or $C(\Bbb T^d)$ (see \cite[Theorem 3.5]{T}).
 According to \cite[Theorem 3.1]{CKS}, we know that the condition $g({\bf R})>1/2$ is sufficient and necessary condition for the
 embedding from  $W_2^{\bf R}(\Bbb T^d)$ to $L_\infty(\Bbb T^d)$ or $C(\Bbb T^d)$. Furthermore, for $g({\bf R})>1/2$, using the proof technique of \cite[Theorem 4.3]{T} we
 can show
\begin{equation}\label{2.7-11}\lim_{n\rightarrow \infty}n^{g({\bf R})-\frac12}a_n(I_d\::
W_2^{\bf R}(\Bbb T^d)\rightarrow L_\infty(\Bbb T^d))= (2g({\bf
R})-1)^{-\frac12}(\mathrm{vol}(B_{2{\bf R}}^d))^{g({\bf
R})}.\end{equation} Note that \eqref{2.7-11} holds if we replace
$L_\infty(\Bbb T^d)$ with $C(\Bbb T^d)$.
   \end{rem}

\begin{rem} When $R_1=R_2=\cdots=R_d=s>0$, \eqref{2.6-11} recedes to \eqref{2.6-111} with $r=2s$.  One can rephrase \eqref{2.6-11} and \eqref{2.7-11} as  strong
equivalences $$a_n\big(I_d :W_2^{\bf R}(\Bbb T^d)\rightarrow L_2
(\Bbb T^d)\big)\sim   n^{-g({\bf R})}(\mathrm{vol}(B_{2{\bf
R}}^d))^{g({\bf R})}
$$ and \begin{align*} a_n\big(I_d :W_2^{\bf R}(\Bbb
T^d)\rightarrow L_\infty (\Bbb T^d)\big)&\sim  (2g({\bf
R})-1)^{-\frac12} n^{-g({\bf R})+\frac12}(\mathrm{vol}(B_{2{\bf
R}}^d))^{g({\bf R})}\\ &\sim  (2g({\bf R})-1)^{-\frac12}
n^{\frac12} a_n\big(I_d :W_2^{\bf R}(\Bbb T^d)\rightarrow L_2(\Bbb
T^d)\big).\end{align*} The novelty of Theorem \ref{t2.2} and
\eqref{2.7-11} is that they give strong equivalences and provide
asymptotically optimal constants, for arbitrary fixed $d$ and
${\bf R}$.
\end{rem}

\begin{thm}\label{t2.2-1}We have
\begin{align}\label{2.9-11}
&a_n(I_d:\ W_2^{\infty}(\Bbb T^d)\rightarrow L_2(\Bbb T^d))\asymp
\left\{\begin{matrix}
1, & \ \  1\le n\leq d,\\
 \Big(\frac{\log(1+\frac{d}{\log n})}{\log n}\Big)^{1/2},&\ \  d\le n\le
 3^d,
\end{matrix}\right.
\end{align}and $a_n(I_d:\ W_2^{\infty}(\Bbb T^d)\rightarrow L_2(\Bbb T^d))=0$ for $n>3^d$, where the equivalent constants do not depend on $n$ and
$d$.
\end{thm}

\begin{thm}\label{t2.3}
Let ${\bf R}=(R_1,...,R_d)\in \Bbb R_+^d$, $\inf\limits_{1\le
j<\infty}R_j>0$, and $\alpha,\,\beta>0$. Then the approximation
problems
$$I_d\:: W_2^{\bf R}(\Bbb T^d)\rightarrow L_2(\Bbb T^d)\ \ {\rm and}\ \ I_d\:: W_2^{\infty}(\Bbb T^d)\rightarrow L_2(\Bbb T^d)$$
are $(\alpha, \beta)$-weakly tractable if and only if $\alpha>2$
and $\beta>0$ or  $\alpha>0$ and $\beta>1$. Specially, the
approximation problems $I_d\:: W_2^{\bf R}(\Bbb T^d)\rightarrow
L_2(\Bbb T^d)$ and $I_d\:: W_2^{\infty}(\Bbb T^d)\rightarrow
L_2(\Bbb T^d)$ are  intractable and do not suffer from the curse
of dimensionality.
\end{thm}

\begin{rem}In \cite{KSU1}, the authors considered the tractability of the
isotropic Sobolev embeddings
\begin{equation}\label{e1.2}I_d:H^{s,r}(\Bbb T^d)\rightarrow
L_2(\Bbb T^d)
\end{equation} for the cases $r=1,\ r=2$, and $r=2s$, and obtained the
following results.

(1) For every $s > 0$,  none of the above mentioned approximation
problems
 \eqref{e1.2} is quasi-polynomially
tractable, and suffers from the curse of dimensionality.

(2) For  $r=1$, \eqref{e1.2} is weakly tractable if $s>1$,
 intractable if $0<s\leq 1$.

(3) For $r=2s$, \eqref{e1.2} is intractable for all $s>0$.

(4) For  $r=2$, \eqref{e1.2} is weakly tractable if $s>2$,
intractable if $0<s\leq 1$, and remains open if $1<s\le 2$.
Specially, the classical Sobolev embedding $I_d : H^{m}(\Bbb
T^d)\to L_2(\Bbb T^d)$ is weakly tractable if $m\ge 3$,
intractable if $m=1$, and  remains open if $m=2$.

 We remark that the above open problem was solved
already in several papers \cite{SW}, \cite{KMU}, and \cite{WW}
(via  a different technique).

In \cite[Theorem 4.1]{SW}, the authors obtained the
$(\alpha,\beta)$-weak tractability of the approximation problem
\eqref{e1.2} with $r=1, r=2s$, and $r=2$. Combining with
\cite[Remark 7.6]{KMU}, we can see easily that  the approximation
problem \eqref{e1.2} is $(\alpha,\beta)$-weakly tractable for
$\az>r/s$ and $\beta
> 0$ or $\az>0$ and $\beta>1$.

 \end{rem}

\begin{rem}It is an interesting problem about the tractability of the classical anisotropic Sobolev embedding problem $I_d\:: H_2^{\bf m}(\Bbb T^d)\rightarrow
L_2(\Bbb T^d),\ {\bf m}\in \Bbb N_+^d$ for $k=\sup_{1\le
j<\infty}m_j<+\infty$. Since both the embeddings $I_d: H^k(\Bbb
T^d)\to H_2^{\bf m}(\Bbb T^d)$ and  $I_d:  H_2^{\bf m}(\Bbb
T^d)\to H^1(\Bbb T^d)$ have the norm $1$,  it follows from
\cite{KSU1, KMU} that the approximation problem $I_d\:: H_2^{\bf
m}(\Bbb T^d)\rightarrow L_2(\Bbb T^d)$ is not uniformly weakly
tractable and does not suffer from the curse of dimensionality.
Concerning with the weak tractability, we conjecture that the
approximation problem $I_d\:: H_2^{\bf m}(\Bbb T^d)\rightarrow
L_2(\Bbb T^d)$ is weakly tractable if $\liminf\limits_{d\to\infty}
d\, g({\bf m})>2$, and intractable if $\liminf\limits_{d\to\infty}
d\, g({\bf m})\le 2$. It is an open problem.
\end{rem}

\section{Proofs of Theorems \ref{t2.1}, \ref{t2.2}, \ref{t2.2-1}, and \ref{t2.3}}

For  ${\bf R}=(R_1,...,R_d)\in \Bbb R_+^d$, let $\{\omega _{{\bf
R},d}^*({l})\}_{l=1}^\infty$ be the non-increasing rearrangement
of
$$\Big\{(1+\sum_{j=1}^{d}|k_j|^{2R_j})^{-1/2}\Big\}_{{\bf k}\in \Bbb
Z^d}.$$According to the results about approximation numbers of
diagonal operators (see e.g. \cite[Lemma 2.4]{KSU1}, or
\cite[Corollary 4.12]{NW1}), we get
$$a_{n}(I_d:W_2^{\bf R}(\Bbb T^d)\to L_2(\Bbb T^d))=\omega _{{\bf R},d}^*(n).$$

For $m\in \Bbb N$, denote by $C(m,{\bf R},d)$ the cardinality of
the set
$$\Big\{\k \ \big|\ \sum_{j=1}^{d}|k_j|^{2R_j}\le m^{2p},\ {\bf k}\in \Bbb
Z^d\Big\},$$where $p=\max\{1/2,\,R_1,\dots,R_d\}$. Then  for
$C(m-1,{\bf R},d)< n\le C(m,{\bf R},d),$ we have
$$(1+m^{2p})^{-1/2}\le a_{n}(I_d : W_2^{\bf R}(\Bbb T^d)\rightarrow L_2(\Bbb T^d))\le (1+(m-1)^{2p})^{-1/2}.$$

The following lemma gives the relation of $\mathrm{vol}(B^d_{\bf
R}(t))\ (t>0)$ and $\mathrm{vol}(B^d_{\bf R})$.
\begin{lem}
Let ${\bf R}=(R_1,...,R_d)\in \Bbb R_+^d,\ t>0$. Then
\begin{equation}\label{4.1}
\mathrm{vol}(B^d_{\bf
R}(t))=t^{1/R_1+...+1/R_d}\,\mathrm{vol}(B^d_{\bf R}).
\end{equation}
\end{lem}
\begin{proof}We make a change of variables
$$y_1=x_1 t^{-1/R_1},\ \dots,\ y_d=x_d\, t^{-1/R_d}$$ that deforms $B_{\bf R}^d(t)$ into $B_{\bf R}^d$. The Jacobian determinant is $J(\y)=t^{1/R_1+...+1/R_d}$.
By the change of variables formula we obtain
$$\mathrm{vol}(B^d_{\bf R}(t))=\int_{B^d_{\bf R}(t)}1\,d\x=\int_{B^d_{\bf R}}J(\y)
\,d\y=t^{1/R_1+...+1/R_d}\,\mathrm{vol}(B^d_{\bf R}), $$which completes the proof.\end{proof}

\begin{lem}\label{l4.2}
Let ${\bf R}=(R_1,...,R_d)\in \Bbb R_+^d,\ p=\max\{\frac{1}{2},\ R_1,\dots,R_d\}$. Then for any $x,\ y\in \Bbb R^d$, we have
\begin{equation}\label{4.2}
\big(\sum_{j=1}^d|x_j+y_j|^{2R_j}\big)^{1/(2p)}\le \big(\sum_{j=1}^d|x_j|^{2R_j}\big)^{1/(2p)}+\big(\sum_{j=1}^d|y_j|^{2R_j}\big)^{1/(2p)}
\end{equation}
\end{lem}
\begin{proof}
Using the inequality $$(a+b)^q\leq a^q+b^q,\ \ a,b\ \ge \ 0,\ 0<
q\ \leq 1,$$ we get that for $1\le j\le d$,
\begin{equation}\label{4.3}
\begin{split}
|x_j+y_j|^{2R_j}&\le \big((|x_j|+|y_j|)^{R_j/p}\big)^{2p}\le
(|x_j|^{R_j/p}+|y_j|^{R_j/p})^{2p}.
\end{split}
\end{equation}
It follows from \eqref{4.3} and the Minkowskii inequality that
\begin{align*}
\big(\sum_{j=1}^d|x_j+y_j|^{2R_j}\big)^{1/(2p)}&\le \big(\sum_{j=1}^d(|x_j|^{R_j/p}+|y_j|^{R_j/p})^{2p}\big)^{1/(2p)}\\
&\le \big(\sum_{j=1}^d|x_j|^{2R_j}\big)^{1/(2p)}+\big(\sum_{j=1}^d|y_j|^{2R_j}\big)^{1/(2p)}.
\end{align*}
Lemma \ref{l4.2} is proved.
\end{proof}

\begin{lem}(\cite{KSU1})\label{l3.3}
For $0\leq x<\infty $, it holds
\begin{align}\label{3.2}
\Big(\frac{x}{e}\Big)^x\leq \Gamma (1+x)\leq (1+x)^x.
\end{align}
\end{lem}
\ Note that $dg({\bf R})=\frac{d}{1/R_1+...+1/R_d}$ is just the
harmonic average of the $d$ positive numbers $R_1,\dots,R_d$ and
is between $R_1,\dots,R_d$ .
\begin{lem}\label{lem3.4}
Let ${\bf R}=(R_1,...,R_d)\in \Bbb R_+^d$,  $u=\max\{ R_1,\dots,R_d\}$, $v=\min\{ R_1,\dots,R_d\}$. Then for all $d\in \Bbb N_+$, we have
\begin{equation}\label{3.5}
2^v\sqrt{\frac{1}{e(d+2u)}}\leq (\mathrm{vol}(B_{2{\bf R}}^d))^{g({\bf R})}\leq 2^u\big(\frac{2v+1}{2v}\big)^{\frac{u}{2v}}\sqrt{\frac{2eu}{d}},
\end{equation}
\end{lem}
\noindent where $g({\bf R})=\frac{1}{1/R_1+...+1/R_d}$, $v\leq
dg({\bf R})\leq u$.
\begin{proof}
By \eqref{2.9}, we get \begin{align*} \Big(\text{vol}(B^d_{2{\bf
R}})\Big)^{g({\bf R})}&=2^{dg({\bf R})}\frac{\Big(\Gamma(1 +
\frac{1}{2R_1})\,\cdots\,\Gamma(1 + \frac{1}{2R_d})\Big)^{g({\bf
R})}}{\Big(\Gamma(1 + \frac{1}{2R_1}+\cdots+\frac{1}{2R_d})
\Big)^{g({\bf R})}}
\\ &=2^{dg({\bf R})} {\Big(\Gamma(1 +
\frac{1}{2R_1})\,\cdots\,\Gamma(1 + \frac{1}{2R_d})\Big)^{g({\bf
R})}} {\Big(\Gamma (1 + \frac{1}{2g({\bf R})})\Big)^{-g({\bf R})}}
.
\end{align*}
From Lemma \ref{l3.3}, we have
\begin{align}\label{3.6}
\Big(\Gamma \big(1 + \frac{1}{2g({\bf R})}\big)\Big)^{-g({\bf R})}
\leq \Big (\frac{1}{2eg({\bf R})}\Big
)^{-\frac{1}{2}}=\Big(\frac{2edg({\bf R})}{d}\Big)^{\frac{1}{2}},
\end{align}
and
\begin{align}\label{3.7}
\Big(\Gamma \big(1 + \frac{1}{2g({\bf R})}\big)\Big )^{-g({\bf
R})}\geq \Big (1 + \frac{1}{2g({\bf R})}\Big )^{-\frac{1}{2}}
=\Big (\frac{2dg({\bf R})}{d + 2dg({\bf R})}\Big)^{\frac{1}{2}}.
\end{align}
It follows from Lemma \ref{l3.3} and the monotonicity of the
function $g(x)=(1+x)^x,\ x\geq 0$ that
\begin{align}\label{3.8}
\Big(\Gamma(1 + \frac{1}{2R_1})\cdots\Gamma(1 + \frac{1}{2R_d})\Big )^{g({\bf R})}&\leq
\Big((1 + \frac{1}{2R_1})^{\frac{1}{2R_1}}\cdots(1 + \frac{1}{2R_d})^{\frac{1}{2R_d}}\Big)^{g({\bf R})}\notag\\
&\leq (1+\frac{1}{2v})^{\frac{dg({\bf R})}{2v}}.
\end{align}
Next we use the convexity of  the function $\ln \Gamma (x), \ x>0$
to get  that  for $x_1,\dots,x_d\geq 0$,
$$\Gamma (1+x_1)\cdots \Gamma (1+x_d)\geq \big(\Gamma (1+\frac{x_1+\dots+x_d}{d})\big )^d.$$
It follows that
\begin{align}\label{3.9}
\Big(\Gamma(1 + \frac{1}{2R_1})\cdots\Gamma(1 +
\frac{1}{2R_d})\Big )^{g({\bf R})}&\geq \Big(\Gamma (1 +
\frac{1}{2dg({\bf R})})\Big)^{dg({\bf R})}\geq \big (\frac
{1}{2edg({\bf R})}\big )^{\frac{1}{2}}.
\end{align}
Then \eqref{3.6} and  \eqref{3.8} lead to
$$(\text{vol}(B^d_{2{\bf R}}))^{g({\bf R})}\leq 2^{dg({\bf R})}(1+\frac{1}{2v})^{\frac{dg({\bf R})}{2v}}\big(\frac{2edg({\bf R})}{d}\big)^{\frac{1}{2}},$$
and  \eqref{3.7} and \eqref{3.9} yield
$$(\text{vol}(B^d_{2{\bf R}}))^{g({\bf R})}\geq 2^{dg({\bf R})}\big(\frac{1}{e(d+2dg({\bf R}))}\big)^{\frac{1}{2}}.$$
Hence,     \eqref{3.5} follows from the two above inequalities and
$v\leq dg({\bf R})\leq u$ immediately.   Lemma \ref{lem3.4} is
proved.
\end{proof}

\begin{lem}\label{l3.5}
Let ${\bf R}=(R_1,...,R_d)\in \Bbb R_+^d$,
$u=\max\{R_1,\dots,R_d\}$, $\ v=\min\{ R_1,\dots,R_d\}$,
$p=\max\{\frac{1}{2},\ u\}$, $g({\bf
R})=\frac{1}{1/R_1+...+1/R_d}$.
 Then
 for $n> E^d$,  we have
\begin{equation} \label{2.4}n^{g({\bf R})}a_n(I_d\:: W_2^{\bf R}(\Bbb T^d)\rightarrow
L_2(\Bbb T^d))\leq
2^{(p+u)}(1+\frac{2v+1}{2v})^{\frac{u}{2v}}\sqrt{\frac{2eu}{d}},\end{equation}
and  \begin{equation}\label{2.5}n^{g({\bf R})}a_n(I_d\:: W_2^{\bf
R}(\Bbb T^d)\rightarrow L_2(\Bbb T^d))\geq
2^{(v-p)}\sqrt{\frac{1}{e(d+2u)}},\end{equation}where $E:=
4^{p/v}2^{u/v}(1+1/(2v))^{u/(2v^2)}(2ep)^{1/(2v)}$ is a positive
constant depending only on $u$ and $v$.
\end{lem}

\begin{proof}
For any $m\in \Bbb N_+$, let $Q_{\bf k}$ be a cube with center
${\bf k}$, sides parallel to the axes and side-length 1. It
follows from \eqref{4.2} that
\begin{align*}
&B_{2{\bf
R}}^d\Big(\big(m-\big(\sum_{j=1}^{d}\big(\frac{1}{2}\big)^{2R_j}\big)^{1/(2p)}\big)^{2p}_+\Big)\subset
\bigcup_{\substack{{\bf k}\in\Bbb Z^d\\\sum_{j=1}^{d}|k_j|^{2R_j}\leq m^{2p}}}Q_{\bf k}\\
&\qquad \subset B_{2{\bf
R}}^d\Big(\big(m+\big(\sum_{j=1}^{d}\big(\frac{1}{2}\big)^{2R_j}\big)^{1/(2p)}\big)^{2p}\Big),
\end{align*}where $a_+$ is equal to $a$ if $a\ge 0$ and $0$ if
$a<0$. Note that the volume of the set
$$\bigcup\limits_{\substack{{\bf k}\in\Bbb
Z^d\\\sum_{j=1}^{d}|k_j|^{2R_j}\leq m^{2p}}}Q_{\bf k}$$ is just
$C(m,{\bf R},d)$. By \eqref{4.1} we get
\begin{align}
&\big(m-\big(\sum_{j=1}^{d}\big(\frac{1}{2}\big)^{2R_j}\big)^{1/(2p)}\big)^{p/g({\bf R})}_+\mathrm{vol}(B_{2{\bf R}}^d) \leq C(m,{\bf R},d)\notag\\
&\qquad\leq
\big(m+\big(\sum_{j=1}^{d}\big(\frac{1}{2}\big)^{2R_j}\big)^{1/(2p)}\big)^{p/g({\bf
R})}\mathrm{vol}(B_{2{\bf R}}^d). \label{4.4}
\end{align}
 We set
$$a_n(I_d):=a_n(I_d :W_2^{\bf R}(\Bbb T^d)\rightarrow L_2(\Bbb
T^d))\ \ {\rm and}\ \ b_{\bf
R}:=\big(\sum_{j=1}^{d}\big(\frac{1}{2}\big)^{2R_j}\big)^{1/(2p)}.$$
We also set  $$A(m,{\bf R},d):=(m+b_{{\bf R}})^{p/g({\bf
R})}\mathrm{vol}(B_{2{\bf R}}^d)$$ {\rm and} $$ B(m,{\bf
R},d):=(m-b_{{\bf R}})_+^{p/g({\bf R})}\mathrm{vol}(B_{2{\bf
R}}^d).$$

First we estimate the upper bound for $n^{g({\bf R})}a_n(I_d)$. We
fixed $m_0\in \Bbb N$ such that $1+b_{{\bf R}}\leq m_0\leq
2+b_{{\bf R}}.$ For all $m\geq m_0$, suppose that
$$A(m,{\bf R},d)<n\leq A(m+1,{\bf R},d).$$ By \eqref{4.4} we have
$n>C(m,{\bf R},d)$, which implies $a_n(I_d)\leq
(1+m^{2p})^{-\frac{1}{2}}.$ It follows that
\begin{align*}
n^{g({\bf R})}a_n(I_d)&\leq \frac{(A(m+1,{\bf R},d))^{g({\bf
R})}}{(1+m^{2p})^{\frac{1}{2}}}
=\frac{{(m+1+b_{\bf R})}^{p}(\mathrm{vol}(B_{2{\bf R}}^d))^{g({\bf R})}}{(1+m^{2p})^{\frac{1}{2}}}\\
&\leq \frac{{(m+1+b_{\bf R})}^{p}(\mathrm{vol}(B_{2{\bf
R}}^d))^{g({\bf R})}}{m^{p}}\leq 2^p(\mathrm{vol}(B_{2{\bf
R}}^d))^{g({\bf R})}.
\end{align*}
By \eqref{3.5} we get further
$$n^{g({\bf R})}a_n(I_d)\leq
2^{(p+u)}(\frac{2v+1}{2v})^{\frac{u}{2v}}\sqrt{\frac{2eu}{d}},$$
which proves  \eqref{2.4} for large enough $n>A(m_0,{\bf R},d).$
 Note that $b_{\bf R}\le d^{1/(2p)}$. Using \eqref{3.5} we get that
\begin{align*}A(m_0,{\bf R},d)&\le (2+2b_{\bf R})^{p/g({\bf R})}\,\mathrm{vol}(B_{2{\bf
R}}^d)\le  (4d^{1/(2p)})^{p/g({\bf R})}\,\mathrm{vol}(B_{2{\bf
R}}^d)\\& \le (4d^{1/(2p)})^{p/g({\bf
R})}\Big(2^u(1+1/(2v))^{u/(2v)}(2eu)^{1/2}d^{-1/2}\Big)^{1/g({\bf
R})}\\&\le \Big(4^p2^u
(1+1/(2v))^{u/(2v)}(2ep)^{1/2}\Big)^{d/(dg({\bf R}))}\\&\le
\Big(4^{p/v}2^{u/v}(1+1/(2v))^{u/(2v^2)}(2ep)^{1/(2v)}\Big)^d=E^d.
\end{align*}
Hence, \eqref{2.4} holds for $n> E^d $.

 Next we show the lower bound.  Choose $m_1\in \Bbb N$ such that $2+2b_{\bf R}\leq
m_1\leq 3+2b_{\bf R}$. For $m\geq m_1$ and $$B(m,{\bf R},d)< n\le
B(m+1,{\bf R},d),$$ by \eqref{4.4} we have $n\leq C(m+1,{\bf
R},d),$ which means $a_n(I_d)\geq (1+(m+1)^{2p})^{-\frac{1}{2}}.$
It follows that
\begin{align*}
n^{g({\bf R})}a_n(I_d)&\geq \frac{(B(m,{\bf R},d))^{g({\bf
R})}}{(1+(m+1)^{2p})^{\frac{1}{2}}}
=\frac{{(m-b_{\bf R})}^{p}(\mathrm{vol}(B_{2{\bf R}}^d))^{g({\bf R})}}{(1+(m+1)^{2p})^{\frac{1}{2}}}\\
&\geq \frac{{(m-b_{\bf R})}^{p}(\mathrm{vol}(B_{2{\bf
R}}^d))^{g({\bf R})}}{(m+2)^{p}}\geq
\big(\frac{2+b_{\bf R}}{4+2b_{\bf R}}\big)^p(\mathrm{vol}(B_{2{\bf R}}^d))^{g({\bf R})}\\
&=2^{-p}\mathrm{vol}(B_{2{\bf R}}^d))^{g({\bf R})}\ge
2^{(v-p)}(e(d+2u))^{-1/2},
\end{align*}
where in the last second inequality we used the monotonicity of
the function $\frac{x-b_{\bf R}}{x+2},\ x\ge 0$, and in the last
inequality we used \eqref{3.5}. Hence \eqref{2.5} holds for $n>
B(m_1,{\bf R},d)$. Similarly, we have
\begin{align*}B(m_1,{\bf R},d)&\le (3+b_{\bf R})^{p/g({\bf R})}\,\mathrm{vol}(B_{2{\bf
R}}^d)\le  (4d^{1/(2p)})^{p/g({\bf R})}\,\mathrm{vol}(B_{2{\bf
R}}^d)\\&\le
\Big(4^{p/v}2^{u/v}(1+1/(2v))^{u/(2v^2)}(2ep)^{1/(2v)}\Big)^d=E^d.
\end{align*}
This means that \eqref{2.5} holds for $n> E^d$.
 The proof of Lemma \ref{l3.5} is completed.
\end{proof}

\begin{rem}Let ${\bf R}=(R_1,\dots,R_d)\in\Bbb R_+^d$,  $u=\max\{R_1,\dots,R_d\}$,  $v=\min\{
R_1,\dots,R_d\}$, and  $g({\bf R})=\frac{1}{1/R_1+...+1/R_d}$.
Then for sufficiently large $n$ $(n>E^d)$, the above lemma
provides the two-sides inequalities
\begin{equation}\label{3.13-1}c_{u,v}d^{-1/2}n^{-g({\bf R})}\le a_n(I_d\:: W_2^{\bf
R}(\Bbb T^d)\rightarrow L_2(\Bbb T^d))\le
C_{u,v}d^{-1/2}n^{-g({\bf R})},\end{equation}where the constants
$c_{u,v}, C_{u,v}$ depend only on $u$ and $v$.
  Note that we captured the exact decay rate in $n$ and the
exact order of the constants with respect to $d$.
\end{rem}

\

\noindent{\it Proof of Theorem \ref{t2.1}.}

 We note that both the
embeddings
$$ W_2^{{\bf R}}(\Bbb T^d) \to H^{v,2v}(\Bbb T^d)\ \ \ {\rm and}\ \ \ H^{u,2u}(\Bbb T^d)\to W_2^{{\bf R}}(\Bbb T^d)$$ have norm
1. It follows that $$a_n(I_d :H^{u,2u}(\Bbb T^d)\rightarrow
L_2(\Bbb T^d))\le  a_n(I_d :W_2^{\bf R}(\Bbb T^d)\rightarrow
L_2(\Bbb T^d))\le  a_n(I_d :H^{v,2v}(\Bbb T^d)\rightarrow L_2(\Bbb
T^d)).$$  Noting $$ \frac{\log n} {\log(1+\frac{d}{\log n})}\asymp
d\ \ {\rm for}\ \ 2^d\le n\le 3^d,$$ and
$$n^{-u/d}\asymp_{u,v} n^{-v/d}\asymp_{u,v} n^{-g({\bf
R})}\asymp_{u,v}1 \ \ {\rm for}\ \ 3^d\le n\le E^d,$$ we obtain
from \eqref{2.3.1} that
\begin{align*}\label{2.3.1}
a_n(I_d: H^{u,2u}(\Bbb T^d)\rightarrow L_2(\Bbb T^d))&\asymp_{u,v}
a_n(I_d: H^{v,2v}(\Bbb T^d)\rightarrow L_2(\Bbb T^d))
\\ &\asymp_{u,v}\left\{\begin{matrix}
1, & \ \  1\le n\leq d,\\
 \Big(\frac{\log(1+\frac{d}{\log n})}{\log n}\Big)^{1/2},&\ \  d\le n\le 3^d, \\
 d^{-1/2},&\ \  3^d\le n\le E^d,
\end{matrix}\right.
\end{align*}
where $E$ is a positive constant given in Lemma \ref{l3.5} which
depends only on $u$ and $v$. We have
$$ a_n(I_d :W_2^{\bf R}(\Bbb T^d)\rightarrow
L_2(\Bbb T^d)) \asymp_{u,v}\left\{\begin{matrix}
1, & \ \  1\le n\leq d,\\
 \Big(\frac{\log(1+\frac{d}{\log n})}{\log n}\Big)^{1/2},&\ \  d\le n\le 3^d, \\
 d^{-1/2}n^{-g({\bf R})},&\ \  3^d\le n\le E^d,\end{matrix}\right.$$
which combining with \eqref{3.13-1}, yields \eqref{2.6-1}. Theorem
\ref{t2.1} is proved. $\hfill\Box$

\

\noindent{\it Proof of Theorem \ref{t2.2}.}

 As in the proof of Lemma \ref{l3.5}, we set $$a_n(I_d):=a_n(I_d :W_2^{\bf R}(\Bbb T^d)\rightarrow L_2(\Bbb
T^d))\ \ {\rm and}\ \ b_{\bf
R}:=\big(\sum_{j=1}^{d}\big(\frac{1}{2}\big)^{2R_j}\big)^{1/(2p)}.$$
For $C(m-1,{\bf R},d)< n\le C(m,{\bf R},d)$, we have
$$(1+m^{2p})^{-1/2}\le a_n(I_d)\le (1+(m-1)^{2p})^{-1/2}.$$
It follows that
\begin{equation}\label{4.5}
n\cdot\big(a_n(I_d ))^{1/g({\bf R})}\le C(m,{\bf
R},d)(1+(m-1)^{2p})^{-1/(2g({\bf R}))},\end{equation}and
\begin{equation}\label{4.6}n\cdot\big(a_n(I_d))^{1/g({\bf R})}>C(m-1,{\bf
R},d)(1+m^{2p})^{-1/(2g({\bf R}))}.
\end{equation}

On one side, from \eqref{4.4} and \eqref{4.5} we get
 $$n\cdot\big(a_n(I_d ))^{1/g({\bf R})}\le (1+(m-1)^{2p})^{-1/2g({\bf R})} \big(m+b_{\bf R}\big)^{p/g({\bf R})}\mathrm{vol}(B_{2{\bf R}}^d).$$
 Obviously, it holds that  $$\lim_{m\to \infty}(1+(m-1)^{2p})^{-1/2g({\bf R})} \big(m+b_{\bf R}\big)^{p/g({\bf R})}=1,$$
which implies that $$\lim_{n\to \infty}n^{g({\bf R})}a_n(I_d )\le
(\mathrm{vol}(B_{2{\bf R}}^d))^{g({\bf R})}.$$

On the other side, \eqref{4.4} and \eqref{4.6} lead to
$$n\cdot\big(a_n(I_d ))^{1/g({\bf R})}>(1+m^{2p})^{-1/2g({\bf R})} \big(m-b_{\bf R}\big)_+^{p/g({\bf R})}\mathrm{vol}(B_{2{\bf R}}^d).$$
Since $$\lim_{m\to \infty}(1+m^{2p})^{-1/2g({\bf R})}
\big(m-b_{\bf R}\big)_+^{p/g({\bf R})}=1,$$ we obtain
$$\lim_{n\to \infty}n^{g({\bf R})}a_n(I_d )\ge
(\mathrm{vol}(B_{2{\bf R}}^d))^{g({\bf R})}.$$ Theorem \ref{t2.2}
is proved. $\hfill\Box$

\

\noindent{\it Proof of Theorem \ref{t2.2-1}.}

First we note that $\dim \big(W_2^\infty(\Bbb T^d)\big)=3^d$,
which implies $$a_n(I_d:\ W_2^{\infty}(\Bbb T^d)\rightarrow
L_2(\Bbb T^d))=0$$ for $n>3^d$. Let $\{\oz_d^*(l)\}_{1\le l\le
3^d}$ be the non-increasing rearrangement of
$$\big\{(1+\sum_{j=1}^d|k_j|)^{-1/2}\big\}_{\k\in\{-1,0,1\}^d}.$$
Then we have for $1\le n\le 3^d$,
$$a_n(I_d:\ W_2^{\infty}(\Bbb T^d)\rightarrow
L_2(\Bbb T^d))=\oz_d^*(n).$$

For $m=0,1,\dots, d$, denote by $C(m,d)$ and $D(m,d)$ the
cardinalities of the sets
$$\big\{\k\in\{-1,0,1\}^d\ |\ \sum_{j=1}^d|k_j|\le m\big\}\ \ {\rm
and}\ \ \big\{\k\in\{-1,0,1\}^d\ |\ \sum_{j=1}^d|k_j|= m\big\}.$$
Then
$$ a_n(I_d:\ W_2^{\infty}(\Bbb T^d)\rightarrow
L_2(\Bbb T^d))=(1+m)^{-1/2}$$for $C(m-1,d)<n\le C(m,d),\
m=1,\dots,d$.
 It is easy to see that for $m=0,1,\dots,
d,$$$D(m,d)=2^m\binom dm \ \ \ {\rm and}\ \ \
C(m,d)=\sum_{j=0}^mD(j,d)=\sum_{j=0}^m 2^j\binom dj.$$

For $1\le n\le C(2,d)=2d^2+1$, we have
$$  3^{-1/2}\le a_n(I_d:\ W_2^\infty(\Bbb T^d)\rightarrow L_2(\Bbb
T^d))\le 1.$$This means that $$a_n(I_d:\ W_2^\infty(\Bbb
T^d)\rightarrow L_2(\Bbb T^d))\asymp1\asymp \left\{\begin{matrix}
1, & \ \  1\le n\leq d,\\
 \Big(\frac{\log(1+\frac{d}{\log n})}{\log n}\Big)^{1/2},&\ \  d\le n\le
 C(2,d).
\end{matrix}\right.$$

For  $C(m,d)<n\le 3^d, \ d/2\le m\le d$, we have
$$ d/2\le m\le \log C(m,d)\le \log n\le d\log 3,$$
and
\begin{align*} (1+d)^{-1/2}&= a_{3^d}(I_d:\ W_2^{\infty}(\Bbb T^d)\rightarrow L_2(\Bbb T^d))\\&\le
a_{n}(I_d:\ W_2^{\infty}(\Bbb T^d)\rightarrow L_2(\Bbb T^d))\\&\le
(2+m)^{-1/2}\le (2+d/2)^{-1/2},\end{align*}which implies that for
$C(m,d)<n\le 3^d, \ d/2\le m\le d$,
$$ a_{n}(I_d:\ W_2^{\infty}(\Bbb T^d)\rightarrow L_2(\Bbb
T^d))\asymp d^{-1/2}\asymp \Big(\frac{\log(1+\frac{d}{\log
n})}{\log n}\Big)^{1/2}.$$

For $C(m-1,d)<n\le C(m,d),\ 3\le m<d/2$, we have
\begin{equation}\label{3.15-11}a_n(I_d:\ W_2^{\infty}(\Bbb T^d)\rightarrow L_2(\Bbb
T^d))=(1+m)^{-1/2}.\end{equation} We note that
$$n\le \sum_{j=0}^m2^j\binom dj\le (m+1)2^m\binom dm\le 2^{2m}e^m(\frac dm)^m, $$
where in the above second inequality we used the inequality
$\binom dj\le \binom dm$ for $0\le j\le m<d/2$, in the above last
inequality we used the inequality (see \cite[(3.6)]{KSU1})
$$\binom{m+d}d\le e^{d-1}(1+m/d)^d\le e^d\Big(\frac{m+d}d\Big)^d.$$
 It follows that $ \log n\le m\log
(4ed/m), $ which implies
$$ m\ge \frac{\log n}{\log (4ed/m)}.$$
Using the inequalities $ \log n\le m\log (4ed/m) $ and $x\ge 2\log
x$ for $x\ge 2$, we obtain
$$\log \Big(\frac{4ed}{\log n}\Big)\ge \log\Big(\frac{4ed}{m\log
(4ed/m)}\Big)=\log(\frac{4ed}m)-\log\Big(\log(\frac
{4ed}m)\Big)\ge \frac12\log(\frac{4ed}m).
$$
This yields \begin{equation}\label{3.16-11} m\ge \frac{\log
n}{2\log (4ed/(\log n))}.\end{equation} On other hand, using the
inequality (see \cite[(3.5)]{KSU1})
$$\binom {m+d}m\ge \max \Big\{\Big(\frac{d+m}m\Big)^m,\,
\Big(\frac{d+m}d\Big)^d\Big\},$$ we have
$$n>C(m-1,d)\ge 2^{m-1}\binom d{m-1}\ge 2^{m-1} (\frac
d{m-1})^{m-1}.$$ This yields
$$m-1\le\frac{ \log n}{\log \big(\frac {2d}{m-1}\big)}\le\log n.$$
It follows that
$$m\le \frac{ \log n}{\log \big(\frac {2d}{\log n}\big)}+1,$$
which combining with \eqref{3.16-11} and \eqref{3.15-11}, leads to
$$ a_n(I_d:\ W_2^{\infty}(\Bbb T^d)\rightarrow L_2(\Bbb
T^d))\asymp m^{-1/2}\asymp \Big( \frac{ \log n}{\log \big(\frac
{2d}{\log
n}\big)}\Big)^{-1/2}\asymp\Big(\frac{\log(1+\frac{d}{\log
n})}{\log n}\Big)^{1/2}.$$ Theorem \ref{t2.2-1} is proved.
$\hfill\Box$

\begin{rem}The upper estimate of $a_n(I_d:\ W_2^{\infty}(\Bbb T^d)\rightarrow
L_2(\Bbb T^d))$ can also be obtained directly by \eqref{2.3.1} and
the inequality
$$a_n(I_d:\ W_2^{\infty}(\Bbb T^d)\rightarrow
L_2(\Bbb T^d))\le a_n(I_d:\ W_2^1(\Bbb T^d)\rightarrow L_2(\Bbb
T^d)).$$\end{rem} \

 \noindent{\it Proof of Theorem
\ref{t2.3}.}

We set $ q=\inf_{1\le j<\infty}R_j>0.$ Both the embeddings
$$ W_2^{{\bf R}}(\Bbb T^d) \to H^{q,2q}(\Bbb T^d)\ \ \ {\rm and}\ \ \ W_2^\infty(\Bbb T^d)\to W_2^{{\bf R}}(\Bbb T^d)$$ have norm 1.

We note  that the approximation problem
 $$I_d: H^{s,2s}(\Bbb T^d)\rightarrow L_2(\Bbb T^d)$$ is
$(\alpha, \beta)$-weakly tractable if $\alpha>0$ and $\beta>1$
(see \cite[Theorem 4.1]{SW}) or $\alpha>2$ and $\beta>0$ (see
\cite[Remark 7.6]{KMU}) for any $s>0$. This means that if
$\alpha>0$ and $\beta>1$  or $\alpha>2$ and $\beta>0$, the
approximation problems $$I_d\:: W_2^{\bf R}(\Bbb T^d)\rightarrow
L_2(\Bbb T^d)\ \ {\rm and}\ \ I_d\:: W_2^{\infty}(\Bbb
T^d)\rightarrow L_2(\Bbb T^d)$$ are  $(\alpha, \beta)$-weakly
 tractable.

On the other hand, it suffices to prove  the $(\alpha,
\beta)$-weak intractability of the approximation problems $I_d\::
W_2^{\infty}(\Bbb T^d)\rightarrow L_2(\Bbb T^d)$ for  $0<\alpha\le
2$ and $0<\beta\le1$.

 We note that
$$e(n,d)=a_{n+1}(I_d:\ W_2^{\infty}(\Bbb T^d)\rightarrow L_2(\Bbb
T^d)).$$This means that $e(3^d-1,d)=(1+d)^{-1/2}$ and $e(n,d)=0$
for $n\ge 3^d$. Choose
$$\vz=\vz_d=(2+d)^{-1/2}.$$  Then $$n(\vz_d,d)=\inf\{n\in\Bbb N\ |\ e(n,d)\le \vz_d\}=3^d.$$If $0<\alpha\le 2$ and $0<\beta\le1$,
then we have $$\lim_{1/\vz_d+d\to\infty}\frac{\ln
(n(\vz_d,d))}{(\vz_d)^{-\alpha}+d^\beta}\ge \lim_{d\to\infty}\frac
{d\ln 3}{d+2+d}=\frac{\ln3}2\neq0,$$ which implies that the
approximation problems $I_d\:: W_2^{\infty}(\Bbb T^d)\rightarrow
L_2(\Bbb T^d)$ is not $(\alpha, \beta)$-weakly tractable if
$0<\alpha\le 2$ and $0<\beta\le1$.

 Hence,   the
approximation problems $$I_d\:: W_2^{\bf R}(\Bbb T^d)\rightarrow
L_2(\Bbb T^d)\ \ {\rm and}\ \ I_d\:: W_2^{\infty}(\Bbb
T^d)\rightarrow L_2(\Bbb T^d)$$ are  $(\alpha, \beta)$-weakly
 tractable
 if and only if $\alpha>0$ and $\beta>1$
 or $\alpha>2$ and $\beta>0$.   The
proof of  Theorem \ref{t2.3} is finished. $\hfill\Box$

\end{document}